\documentclass[12pt,a4paper]{article}

\usepackage{amsmath,amsthm,amssymb,color}
\usepackage[authoryear]{natbib}
\usepackage{booktabs}

\usepackage{bbm}
\usepackage{mathrsfs}
\usepackage[breaklinks=true]{hyperref}
\usepackage{graphicx}
\usepackage{tikz}
\usetikzlibrary{arrows, automata}
\usetikzlibrary{calc}
\usetikzlibrary{positioning}

\numberwithin{equation}{section}
\allowdisplaybreaks[4]

\theoremstyle{plain}
\newtheorem{theorem}{Theorem}[section]
\newtheorem{proposition}{Proposition}[section]
\newtheorem{lemma}{Lemma}[section]
\newtheorem{corollary}{Corollary}[section]

\newtheorem{conjecture}{Conjecture}[section]

\theoremstyle{definition}
\newtheorem{definition}{Definition}[section]

\newtheorem{remark}{Remark}[section]

\makeatletter

\newcount\minute
\newcount\hour
\newcount\hourMins
\def\now{%
\minute=\time%
\hour=\time \divide \hour by 60%
\hourMins=\hour \multiply\hourMins by 60%
\advance\minute by -\hourMins%
\zeroPadTwo{\the\hour}:\zeroPadTwo{\the\minute}%
}
\def\zeroPadTwo#1{\ifnum #1<10 0\fi#1}

\renewcommand{\cite}{\citet}

\def\^#1{\ifmmode {\mathaccent"705E #1} \else {\accent94 #1} \fi}
\def\~#1{\ifmmode {\mathaccent"707E #1} \else {\accent"7E #1} \fi}

\def\*#1{#1^\ast}
\edef\-#1{\noexpand\ifmmode {\noexpand\bar{#1}} \noexpand\else \-#1\noexpand\fi}
\def\>#1{\vec{#1}}
\def\.#1{\dot{#1}}

\def\atop{\@@atop}
\def\*#1{\mathscr{#1}}

\renewcommand{\leq}{\leqslant}
\renewcommand{\geq}{\geqslant}

\newcommand{\eq}{\eqref}

\newcommand{\N}{\mathcal{N}}

\newcommand{\IE}{\mathbbm{E}}
\newcommand{\IP}{\mathbbm{P}}
\newcommand{\Var}{\mathop{\mathrm{Var}}\nolimits}

\def\be#1{\begin{equation*}#1\end{equation*}}
\def\ben#1{\begin{equation}#1\end{equation}}
\def\bes#1{\begin{equation*}\begin{split}#1\end{split}\end{equation*}}
\def\besn#1{\begin{equation}\begin{split}#1\end{split}\end{equation}}

\def\floor#1{{\lfloor#1\rfloor}}

\def\beqn#1\eeqn{\begin{align}#1\end{align}}
\def\beq#1\eeq{\begin{align*}#1\end{align*}}

\usepackage{graphicx}
\usepackage{latexsym}
\usepackage{amsmath,amsthm,amssymb,amscd}
\usepackage{epsf,amsmath}

\def\E{{\IE}}
\def\P{{\IP}}

\renewcommand\section{\@startsection {section}{1}{\z@}%
{-3.5ex \@plus -1ex \@minus -.2ex}%
{1.3ex \@plus.2ex}%
{\center\small\sc\mathversion{bold}\MakeUppercase}}

\def\subsection#1{\@startsection {subsection}{2}{0pt}%
{-3.5ex \@plus -1ex \@minus -.2ex}%
{1ex \@plus.2ex}%
{\bf\mathversion{bold}}{#1}}

\def\subsubsection#1{\@startsection{subsubsection}{3}{0pt}%
{\medskipamount}%
{-10pt}%
{\normalsize\itshape}{\kern-2.2ex. #1.}}

\def\blfootnote{\xdef\@thefnmark{}\@footnotetext}

\makeatother

\begin{document}

\title{Wasserstein-2 bounds in normal approximation under local dependence}
\author{Xiao Fang}
\date{\it The Chinese University of Hong Kong} %\\[2ex]  \today %, \now}
\maketitle

\noindent{\bf Abstract:} 
We obtain a general bound for the Wasserstein-2 distance in normal approximation for sums of locally dependent random variables.
The proof is based on an asymptotic expansion for expectations of second-order differentiable functions of the sum.
We apply the main result to obtain Wasserstein-2 bounds in normal approximation for sums of $m$-dependent random variables, U-statistics and subgraph counts in the Erd\H{o}s-R\'enyi random graph.
We state a conjecture on Wasserstein-$p$ bounds for any positive integer $p$ and provide supporting arguments for the conjecture.
%Applied to subgraph counts in the Erd\H{o}s-R\'enyi random graph, our result shows that the Wasserstein-1 bound of Barbour, Karo\'nski and Ruci\'nski (1989) holds for the stronger Wasserstein-2 distance.
%We make two conjectures regarding the Wasserstein-$p$ distance for general $p>1$.

\medskip



%\begin{keyword}[class=AMS]
%\kwd[Primary ]{60F05,62E20} \kwd[; secondary ]{62L10}
%\end{keyword}

\section{Introduction}

For two probability measures $\mu$ and $\nu$ on $\mathbb{R}^d$, the so-called Wasserstein-$p$ distance, $p\geq 1$, is defined as
\be{
\mathcal{W}_p(\mu, \nu)=\Big( \inf_{\pi\in \Gamma(\mu,\nu)} \int |x-y|^p d\pi(x,y)  \Big)^{\frac{1}{p}},
}
where $\Gamma(\mu,\nu)$ is the space of all probability measures on $\mathbb{R}^d\times \mathbb{R}^d$ with $\mu$ and $\nu$ as marginals and $|\cdot|$ denotes the Euclidean norm.
Note that $\mathcal{W}_p(\mu, \nu)\leq \mathcal{W}_q(\mu, \nu)$ if $p\leq q$.
For a random vector $W$ whose distribution is close to $\nu$, it is of interest to provide an explicit upper bound on their Wasserstein-$p$ distance.
See, for example, \cite{LeNoPe15}, \cite{Bob18}, \cite{Zh18}, \cite{Bon18} and \cite{CoFaPa18} for a recent wave of research in this direction.

We consider the central limit theorem in dimension one where $\mu$ is the distribution of a random variable $W$ of interest, 
$\nu=N(0,1)$ and $d=1$ in the above setting.
A large class of random variables that can be approximated by a normal distribution exhibits a \emph{local dependence} structure. 
Roughly speaking, with details deferred to Section 2.1, we assume that the random variable $W$ is a sum of a large number of random variables $\{X_i: i\in I\}$ and that each $X_i$ is independent of $\{X_j: j\notin A_i\}$ for a relatively small index set $A_i$. 
\cite{BaKaRu89} obtained a Wasserstein-1 bound in the central limit theorem for such $W$ and \cite{ChSh04} obtained a bound for the Kolmogorov distance. We refer to these two papers for a number of interesting applications.

To prove their Wasserstein-1 bound, \cite{BaKaRu89} used Stein's method and the following equivalent definition of the Wasserstein-1 distance:
\be{
\mathcal{W}_1(\mu,\nu)=\sup_{h\in \text{Lip}_1(\mathbb{R})} \Big| \int_\mathbb{R} h d\mu-\int_\mathbb{R} h d\nu\Big|,
}
where Lip$_1(\mathbb{R})$ denotes the class of Lipschitz functions with Lipschitz constant 1.
There seems to be no such expression for $\mathcal{W}_p$ for general $p$.
The optimal Wasserstein-$p$ bound in normal approximation for sums of independent random variables (cf. Lemma \ref{l3}) was only recently obtained by \cite{Bob18} using characteristic functions.
Our main result, Theorem \ref{t1}, provides a Wasserstein-2 bound in normal approximation under local dependence,
which is a generalization of independence.
We also state a conjecture on Wasserstein-$p$ bounds for any positive integer $p$.

To prove our main result,
we follow the approach of \cite{Ri09}, who used the asymptotic expansion of \cite{Ba86} and a Poisson-like approximation to obtain a Wasserstein-2 bound in normal approximation for sums of independent random variables.
We first use Stein's method to obtain an asymptotic expansion for expectations of second-order differentiable functions of the sum of locally dependent random variables $W$. We then use this expansion and the upper bound for the Wasserstein-2 distance in terms of Zolotarev's ideal distance of order 2 to control the Wasserstein-2 distance between the distributions of $W$ and a sum of independent and identically distributed (i.i.d.) random variables. Finally, we use the triangle inequality and known Wasserstein-2 bounds in normal approximation for sums of i.i.d.\ random variables to prove our main result. 
This approach enables us to potentially bound the Wasserstein-$p$ distance for any positive integer $p$.

We apply our main result to the central limit theorem for sums of $m$-dependent random variables, U-statistics and subgraph counts in the Erd\H{o}s-R\'enyi random graph.

The paper is organized as follows. Section 2 contains the Wasserstein-2 bound in normal approximation under local dependence, the applications and the conjecture on Wasserstein-$p$ bounds.
Section 3 contains some related literature, the proofs of the results in Section 2 and supporting arguments for the conjecture.
In the following, we use $C$ to denote positive constants independent of all other parameters, possibly different from line to line.

\section{Main results}
In this section, we provide a general Wasserstein-2 bound in normal approximation under local dependence and  
apply it to the central limit theorem for sums of $m$-dependent random variables, U-statistics and subgraph counts in the Erd\H{o}s-R\'enyi random graph.
We also state a conjecture on Wasserstein-$p$ bounds.
\subsection{A Wasserstein-2 bound under local dependence}

Let $W=\sum_{i\in I} X_i$ for an index set $I$ with $\E X_i=0, \E W^2=1$ and satisfies the following \emph{local dependence} structure:
\begin{enumerate}
\item[(LD1):] For each $i\in I$, there exists $A_i\subset I$ such that $X_i$ is independent of $\{X_j: j\notin A_i\}$.
\item[(LD2):] For each $i\in I$ and $j\in A_i$, there exists $A_{ij}\supset A_i$ such that $\{X_i, X_j\}$ is independent of $\{X_k: k\notin A_{ij}\}$.
\item[(LD3):] For each $i\in I$, $j\in A_i$ and $k\in A_{ij}$, there exists $A_{ijk}\supset A_{ij}$ 
such that $\{X_i, X_j, X_k\}$ is independent of $\{X_l: l\notin A_{ijk}\}$.
\end{enumerate}
\begin{theorem}\label{t1}
Under the above setting, we have
\ben{\label{2}
\mathcal{W}_2(\mathcal{L}(W), N(0,1))\leq C\big[|\beta|+(\gamma_1+\gamma_2+\gamma_3)^{\frac{1}{2}} \big],
}
where
\be{
\beta=\sum_{i\in I} \sum_{j,k\in A_i} \E X_i X_j X_k +2\sum_{i\in I} \sum_{j\in A_i}\sum_{k\in A_{ij}\backslash A_i} \E X_i X_j X_k,
}
\be{
\gamma_1=\sum_{i\in I} \sum_{j\in A_i}\sum_{k\in A_{ij}}\sum_{l\in A_{ijk}} \E |X_i X_j X_k X_l|,
}
\be{
\gamma_2=\sum_{i\in I} \sum_{j\in A_i}\sum_{k\in A_{ij}}\sum_{l\in A_{ijk}} \E |X_i X_j| \E|X_k X_l|,
}
\be{
\gamma_3=\sum_{i\in I} \sum_{j\in A_i}\sum_{k\in A_{ij}}\sum_{l\in A_{ijk}} \E |X_i X_j X_k| \E |X_l|.
}
\end{theorem}
\begin{remark}
The conditions (LD1)--(LD3) and the bound \eq{2} is a natural extension of (2.1)--(2.5) and (2.7) of \cite{BaKaRu89}.
The sizes of \emph{neighborhoods} $A_{ij}$ and $A_{ijk}$ are typically smaller than those used in \cite{ChSh04}.
It would be interesting to prove a bound for the Kolmogorov distance under the above setting.
%Although we only provide a Wasserstein-2 bound, our approach may work for Wasserstein-$p$ distances for general $p>1$.
%See Section 3.1.4.
\end{remark}
%We have the following immediate corollary. Although the corollary does not yield the optimal bound when applied to the example in %Section 3, it may be of interest for its simplicity.
%\begin{corollary}
%Under the above setting, if we assume further that the cardinalities of $A_i$, $A_{ij}$, $A_{ijk}$ are all bounded by $b$, then
%\be{
%\mathcal{W}_2(\mathcal{L}(W), N(0,1))\leq C\sum_{p=1}^2 \big[b^{1+p} \sum_{i\in I} \E |X_i|^{2+p} \big]^{1/p}
%}
%\end{corollary}

%Conjecture 1

\subsection{Applications}

\subsubsection{$m$-dependence}

Let $X_1,\dots, X_n$ be a sequence of $m$-dependent random variables, namely, $\{X_i: i\leq j\}$ is independent of $\{X_i: i\geq j+m+1\}$ for any $j=1,\dots, n-m-1$.
Let $W=\sum_{i=1}^n X_i$.
Assume that $\E X_i=0$ and $\E W^2=1$.
We have the following corollary of Theorem \ref{t1}.
\begin{corollary}\label{cor1}
For sums of $m$-dependent random variables as above, we have
\be{
\mathcal{W}_2 (\mathcal{L}(W), N(0,1))\leq C\Big\{m^2 \sum_{i=1}^n \E |X_i|^3 +m^{3/2} (\sum_{i=1}^n \E X_i^4)^{1/2}\Big\}.
}
\end{corollary}

\subsubsection{U-statistics}

Let $X_1, X_2,\dots$ be a sequence of i.i.d.\ random variables from a fixed distribution.
Let $m\geq 2$ be a fixed integer.
Let $h: \mathbb{R}^m\to \mathbb{R}$ be a fixed, symmetric, Borel-measurable function. 
We consider the \cite{Ho48} U-statistic 
\be{
\sum_{1\leq i_1<\dots <i_m\leq n} h(X_{i_1},\dots, X_{i_m}).
}
Assume that
\be{
\E h(X_1,\dots, X_m)=0,\   \E h^4(X_1,\dots, X_m)<\infty.
}
and the U-statistic is non-degenerate, namely,
\be{
\E g^2(X_1)>0,
}
where
\be{
g(x):=\E(h(X_1,\dots, X_m)|X_1=x).
}
Applying Theorem \ref{t1} to the U-statistic above yields the following result:
\begin{theorem}\label{t3}
Under the above setting, let 
\be{
W_n=\frac{1}{\sigma_n} \sum_{1\leq i_1<\dots <i_m\leq n} h(X_{i_1},\dots, X_{i_m}),
}
where
\be{
\sigma_n^2=\Var \big[  \sum_{1\leq i_1<\dots <i_m\leq n} h(X_{i_1},\dots, X_{i_m})   \big].
}
We have
\be{
\mathcal{W}_2(\mathcal{L}(W_n), N(0,1))\leq \frac{C}{\sqrt{n}} .
%\Big\{ \E|h(X_1,\dots, X_m)|^3 + \big[  \E h^4(X_1,\dots, X_m)  \big]^{1/2}  \Big\}.
}
\end{theorem}

\begin{remark}
\cite{ChSh07} obtained a bound on the Kolmogorov distance in normal approximation for non-degenerate U-statistics. 
We refer to the references therein for a large literature on the rate of convergence in normal approximation for U-statistics.
In principle, we can take into account in our bound of those \emph{fixed} parameters in the above setting. 
However, we prefer to keep it simple and just show the correct rate of convergence in $n$.
\end{remark}

\subsubsection{Subgraph counts in the Erd\H{o}s-R\'enyi random graph}

Let $K(n,p)$ be the Erd\H{o}s-R\'enyi random graph with $n$ vertices. 
Each pair of vertices is connected with probability $p$ and remain disconnected with probability $1-p$, independent of all else.
Let $G$ be a given fixed graph.
For any graph $H$, let $v(H)$ and $e(H)$ denote the number of its vertices and edges, respectively.
Theorem \ref{t1} leads to the following result.
\begin{theorem}\label{t2}
Let $S$ be the number of copies (not necessarily induced) of $G$ in $K(n,p)$, and let $W=(S-\E S)/\sqrt{\Var(S)}$ be the standardized version.
 Then
\ben{\label{1}
\mathcal{W}_2 (\mathcal{L}(W), N(0,1))\leq C(G)
\begin{cases}
\psi^{-\frac{1}{2}} & \text{if}\ 0<p\leq \frac{1}{2}\\
n^{-1}(1-p)^{-\frac{1}{2}} & \text{if}\ \frac{1}{2}<p<1,
\end{cases}
}
where $C(G)$ is a constant only depending on $G$ and 
\be{
\psi=\min_{H\subset G, e(H)>0} \{n^{v(H)}p^{e(H)}\}.
}
\end{theorem}
\begin{remark}
\cite{BaKaRu89} proved the same bound as in \eq{1} for the weaker Wasserstein-1 distance. 
In the special case where $G$ is a triangle, the bound in \eq{1} reduces to
\be{
C
\begin{cases}
n^{-\frac{3}{2}} p^{-\frac{3}{2}} & \text{if}\ 0<p\leq n^{-\frac{1}{2}} \\
n^{-1}p^{-\frac{1}{2}} & \text{if}\  n^{-\frac{1}{2}}<p\leq \frac{1}{2} \\
n^{-1} (1-p)^{-\frac{1}{2}} & \text{if}\ \frac{1}{2}<p<1.
\end{cases}
}
\cite{Ro17} proved the same bound for the Kolmogorov distance in this special case.
\end{remark}

\subsection{Conjecture on Wasserstein-$p$ bounds}

Here we state a conjecture on Wasserstein-$p$ bounds for any positive integer $p$. We provide supporting arguments for the conjecture at the end of the next section.
Let $W=\sum_{i\in I} X_i$ for an index set $I$ with $\E X_i=0, \E W^2=1$ and satisfies (LD1)--(LD$(p+1)$) where
\begin{enumerate}
\item[(LD$m$):] For each $i_1\in I, i_2\in A_{i_1} \dots, i_m\in A_{i_1\dots i_{m-1}}$, there exists $A_{i_1\dots i_m}\supset A_{i_1\dots i_{m-1}}$ such that $\{X_{i_1},\dots, X_{i_m}\}$ is independent of $\{X_j: j\notin A_{i_1\dots i_m}\}$.
\end{enumerate}

\begin{conjecture}\label{con1}
Under the above setting, we have
\ben{\label{24}
\mathcal{W}_p(\mathcal{L}(W), n(0,1))\leq C_p\sum_{m=1}^p (R_m)^{\frac{1}{m}},
}
where $C_p$ is a constant only depending on $p$,
\be{
R_m=\sum_{i_1\in I} \sum_{i_2\in A_{i_1}} \dots \sum_{i_{m+2}\in A_{i_1\dots i_{m+1}}} \sum_{(\E)}
\E |X_{i_1} X_{i_2}| (\E) |X_{i_3}|\cdots (\E) |X_{i_{m+2}}|,
}
and $\sum_{(\E)}$ denotes the sum over a possible $\E$ in front of each $X_i$ with the constraint that any pair of $\E's$ must be separated by at least two $X_i's$.
\end{conjecture}
\begin{remark}
The case $p=1$ was proved by \cite{BaKaRu89}.
For the case $p=2$, we have $R_2=\gamma_1+\gamma_2+\gamma_3$ where $\gamma_1$--$\gamma_3$ are defined as in Theorem \ref{t1}. In this case, the bound in \eq{24} is clearly an upper bound for the bound in \eq{2}.
\end{remark}

\section{Proofs}

\subsection{Preliminaries}
To prepare for the proof of Theorem \ref{t1}, we need the following lemmas.
The first lemma relates Wasserstein-$p$ distances to Zolotarev's ideal metrics.

\begin{definition}\label{d1}
For $p>1$, let $l=\lceil p \rceil -1$ be the largest integer that is smaller than $p$ and $\Lambda_p$ be the class of $l$-times continuously differentiable functions
$f:\mathbb{R}\to \mathbb{R}$ such that $|f^{(l)}(x)-f^{(l)}(y)|\leq |x-y|^{p-l}$ for any $(x,y)\in \mathbb{R}^2$. The ideal distance $Z_p$ of Zolotarev between two probability distributions $\mu$ and $\nu$ is defined by
\be{
Z_p(\mu, \nu)=\sup_{f\in \Lambda_p}\Big\{ \int_{\mathbb{R}} fd\mu -\int_{\mathbb{R}} fd\nu \Big\}.
}
\end{definition}

\begin{lemma}[Theorem 3.1 of \cite{Ri09}]\label{l1}
For any $p>1$ there exists a positive constant $C_p$, such that for any pair $(\mu,\nu)$ of laws on the real line with finite absolute moments of order $p$,
\be{
\mathcal{W}_p (\mu, \nu)\leq C_p \big[Z_p (\mu, \nu) \big]^{\frac{1}{p}}.
}
\end{lemma}

We use Stein's method to obtain the asymptotic expansion \eq{4} in the proof of Theorem \ref{t1}. 
Stein's method was discovered by \cite{St72} to prove central limit theorems. 
The method has been generalized to other limit theorems and drawn considerable interest recently.
We refer to the book by \cite{ChGoSh11} for an introduction to Stein's method.
\cite{Ba86} used Stein's method to obtain an asymptotic expansion for expectations of smooth functions of sums of independent random variables. \cite{RiRo03} considered a related expansion for dependency-neighborhoods chain structures.
%Our expansion can be regarded as a refined version of the expansion of \cite{RiRo03} for the local dependence structure (LD1)--(LD3) in Section 2.1.

For a function $h$, denote $\mathcal{N} h:=\E h(Z)$, where $Z\sim N(0,1)$, provided that the expectation exists.
Consider the Stein equation
\ben{\label{7}
f'(w)- w f(w)=h(w)-\mathcal{N} h.
}
Let
\besn{\label{33}
f_h(w)=& \int_{-\infty}^w e^{\frac{1}{2}(w^2-t^2)} \big\{ h(t)-\N h  \big\} dt  \\
=& -\int_{w}^\infty e^{\frac{1}{2}(w^2-t^2)} \big\{ h(t)-\mathcal{N} h  \big\} dt.
}
We will use the following lemma.
\begin{lemma}[Special case of Lemma 6 of \cite{Ba86}]\label{l2}
For any positive integer $p>1$, let $h\in \Lambda_p$ where $\Lambda_p$ is defined in Definition \ref{d1}.
Then $f_h$ in \eq{33} is a solution to \eq{7}.
Moreover, $f_h$ is $p$ times differentiable, and satisfies
\be{
|f_h^{(p)}(x)-f_h^{(p)}(y)|\leq C_p|x-y|
}
for any $x, y\in \mathbb{R}$.
\end{lemma}

In the final step of the proof of Theorem \ref{t1}, we will invoke the known Wasserstein-2 bounds in the central limit theorem for sums of i.i.d.\ random variables. The following result was recently proved by \cite{Bob18}.

\begin{lemma}[Theorem 1.1 of \cite{Bob18}]\label{l3}
Let $V_n=\sum_{i=1}^n \xi_i$ where $\{\xi_1,\dots, \xi_n\}$ are independent, with $\E \xi_i=0$ and $\E V_n^2=1$. Then for any real $p\geq 1$,
\ben{\label{17}
\mathcal{W}_p(\mathcal{L}(V_n), N(0,1))\leq C_p \big[ \sum_{i=1}^n \E |\xi_i|^{p+2} \big]^{\frac{1}{p}},
}
where $C_p$ continuously depends on $p$.
\end{lemma}
The results for $p\in  (1,2]$ and for $p>1$ but i.i.d.\ case were first proved by \cite{Ri09}, who also showed that the bound in \eq{17} is optimal.
%The case for general $p>1$ where $\{\xi_1,\dots, \xi_n\}$ are independent and identically distributed was first proved by \cite{Bon18}.

\subsection{Proof of Theorem \ref{t1}}

As noted in the Introduction, the proof consists of three steps.
We first obtain an asymptotic expansion for $\E h(W)$ for $h\in \Lambda_2$.
We then use the expansion and Lemma \ref{l1} to control the Wasserstein-2 distance between the distributions of $W$ and a sum of i.i.d.\ random variables. 
Finally, we use the triangle inequality and known Wasserstein-2 bounds in Lemma \ref{l3} for sums of i.i.d.\ random variables to prove our main result.
Without loss of generality, we assume that the right-hand side of \eq{2} is finite.

\subsubsection{Asymptotic expansion for $\E h(W)$}
In this step, we prove the following proposition.
\begin{proposition}\label{p1}
Let $W$ be as in Theorem \ref{t1}, let $h\in \Lambda_2$ and let $f=f_h$ be the solution \eq{33} to the Stein equation
\ben{\label{26}
f'(w)-wf(w)=h(w)-\mathcal{N}h.
} 
We have
\besn{\label{4}
& \Big| \E h(W) -\N h +\frac{\beta}{2} \N f''    \Big|\\
\leq & C\Big[ |\beta| \mathcal{W}_2(\mathcal{L}(W), N(0,1)) +\gamma_1+\gamma_2+\gamma_3    \Big],
}
where $\beta$, $\gamma_1$--$\gamma_3$ are as in Theorem \ref{t1}.
\end{proposition}
\begin{proof}[Proof of Proposition \ref{p1}]
From $h\in \Lambda_2$ and Lemma \ref{l2}, we have
\ben{\label{8}
|f''(x)-f''(y)|\leq C|x-y|
}
for any $x, y\in \mathbb{R}$.
From \eq{26}, we have
\ben{\label{16}
\E h(W)-\N h =\E f'(W) -\E Wf(W).
}
For each index $i\in I$, let 
\be{
W^{(i)}=W-\sum_{j\in A_i} X_j.
}
By (LD1), $X_i$ is independent of $W^{(i)}$.
From $\E X_i=0$, Taylor's expansion and \eq{8}, we have
\besn{\label{9}
&\E Wf(W)=\sum_{i\in I} \E X_i f(W)=\sum_{i\in I} \E X_i [f(W)-f(W^{(i)})]\\
=& \sum_{i\in I} \sum_{j\in A_i} \E X_i X_j f'(W^{(i)}) +\frac{1}{2}\sum_{i\in I}\sum_{j,k\in A_i} \E X_i X_j X_k f''(W^{(i)})+O(\gamma_1),
}
We begin by dealing with the first term on the right-hand side of \eq{9}.
The second term will be dealt with similarly.
%We construct the first set of auxiliary random variables as follows. 
%For each $i\in I$ and $j\in A_i$, let $\{\widetilde{X}_k: k\in A_{ij}\}$ be an independent copy of $\{X_k: k\in A_{ij}\}$ %conditioning on $\{X_l: l\notin A_{ij}\}$. Let $\widetilde{W}^{(ij)}=\sum_{k\notin A_{ij}} X_k+\sum_{k\in A_{ij}}\widetilde{X}_k$.
%Note that $\{\widetilde{X}_k: k\in A_{ij}\}$ and $\widetilde{W}^{(ij)}$ are independent of $\{X_i, X_j\}$.
%Moreover, $\widetilde{W}^{(ij)}$ has the same distribution as $W$.
In (LD2), 
let 
\be{
W^{(ij)}=W-\sum_{k\in A_{ij}}X_k.
}
By the independence of $\{X_i, X_j\}$ and $W^{(ij)}$ and \eq{8},
we have
\bes{
& \E X_i X_j f'(W^{(i)})= \E X_i X_j \E f'(W^{(ij)}) + \E X_i X_j \big[f'(W^{(i)})-f'(W^{(ij)})\big]\\
=& \E X_i X_j \E f'(W) + \E X_i X_j \Big\{ \E\big[f'(W^{(ij)})-f'(W)\big]+  \big[f'(W^{(i)})-f'(W^{(ij)})\big]  \Big\}\\
=& \E X_i X_j \E f'(W) + \E X_i X_j \E \big[-\sum_{k\in A_{ij}} X_k f''(W^{(ij)}) +O(\sum_{k\in A_{ij}} |X_k|)^2      \big]\\
&+\E X_i X_j \big[ \sum_{k\in A_{ij}\backslash A_i} X_k f''(W^{(ij)}) +O(\sum_{k\in A_{ij}} |X_k| )^2    \big].\\
}
By the assumption that $\E W^2=\sum_{i\in I} \sum_{j\in A_i} \E X_i X_j=1$, we have
\be{
\sum_{i\in I}\sum_{j\in A_i} \E X_i X_j \E f'(W)=\E f'(W).
}
Therefore, 
\besn{\label{19}
&\sum_{i\in I}\sum_{j\in A_i} \E X_i X_j f'(W^{(i)})\\
=& \E f'(W) -\sum_{i\in I}\sum_{j\in A_i}\sum_{k\in A_{ij}} \E X_i X_j \E X_k f''(W^{ij})\\
&+ \sum_{i\in I}\sum_{j\in A_i}\sum_{k\in A_{ij}\backslash A_i} \E X_i X_j  X_k f''(W^{ij})+O(\gamma_1+\gamma_2).
}
In (LD3), 
let 
\be{
W^{(ijk)}=W-\sum_{l\in A_{ijk}}X_l.
}
By the independence of $\{X_i, X_j, X_k\}$ and $W^{(ijk)}$, $\E X_k=0$ and \eq{8},
we have
\besn{\label{20}
&\sum_{i\in I}\sum_{j\in A_i}\sum_{k\in A_{ij}} \E X_i X_j \E X_k f''(W^{ij})\\
=& \sum_{i\in I}\sum_{j\in A_i}\sum_{k\in A_{ij}} \E X_i X_j \E X_k \big[f''(W^{ij})-f''(W^{(ijk)}) \big]\\
=&O(\gamma_2).
}
Similarly,
\besn{\label{21}
&\sum_{i\in I}\sum_{j\in A_i}\sum_{k\in A_{ij}\backslash A_i} \E X_i X_j  X_k f''(W^{ij})\\
=&\sum_{i\in I}\sum_{j\in A_i}\sum_{k\in A_{ij}\backslash A_i} \E X_i X_j  X_k \E f''(W^{(ijk)}) \\
& +\sum_{i\in I}\sum_{j\in A_i}\sum_{k\in A_{ij}\backslash A_i} \E X_i X_j  X_k \big[ f''(W^{ij})   -f''(W^{(ijk)}) \big]\\
=& \sum_{i\in I}\sum_{j\in A_i}\sum_{k\in A_{ij}\backslash A_i} \E X_i X_j  X_k \E f''(W) +O(\gamma_1+\gamma_3)
}
Combining \eq{19}, \eq{20} and \eq{21}, we have
\besn{\label{23}
&\sum_{i\in I}\sum_{j\in A_i} \E X_i X_j f'(W^{(i)})\\
=&\E f'(W)+\sum_{i\in I}\sum_{j\in A_i}\sum_{k\in A_{ij}\backslash A_i} \E X_i X_j  X_k \E f''(W) +O(\gamma_1+\gamma_2+\gamma_3).
}
Similar arguments applied to the second term on the right-hand side of \eq{9} yield
\besn{\label{22}
&\frac{1}{2}\sum_{i\in I}\sum_{j,k\in A_i} \E X_i X_j X_k f''(W^{(i)})\\
%=&\frac{1}{2}\sum_{i\in I}\sum_{j,k\in A_i} \E X_i X_j X_k \E f''(W^{(ijk)})+\frac{1}{2}\sum_{i\in I}\sum_{j,k\in A_i} \E X_i X_j X_k \big[f''(W^{(i)})-f''(W^{ijk})   \big]\\
=&\frac{1}{2}\sum_{i\in I}\sum_{j,k\in A_i} \E X_i X_j X_k \E f''(W) \\
&+ \frac{1}{2}\sum_{i\in I}\sum_{j,k\in A_i} \E X_i X_j X_k \Big\{ \E\big[ f''(W^{ijk})-f''(W)  \big] +   \big[f''(W^{(i)})-f''(W^{(ijk)})   \big]  \Big\}\\
=&\frac{1}{2}\sum_{i\in I}\sum_{j,k\in A_i} \E X_i X_j X_k \E f''(W) +O(\gamma_1+\gamma_3).
}
From \eq{16}, \eq{9}, \eq{23} and \eq{22}, we have
\besn{\label{25}
&\E h(W)-\N h=\E f'(W)-\E Wf(W)\\
=& -\sum_{i\in I}\sum_{j\in A_i}\sum_{k\in A_{ij}\backslash A_i} \E X_i X_j X_k  \E f''(W)- \frac{1}{2}\sum_{i\in I}\sum_{j,k\in A_i} \E X_i X_j X_k \E f''(W)\\
&+O(\gamma_1+\gamma_2+\gamma_3)\\
=&-\frac{\beta}{2} \E f''(W) + O(\gamma_1+\gamma_2+\gamma_3).
}
From \eq{8} and the equivalent definition of the Wasserstein-1 distance
\be{
\mathcal{W}_1(\mu, \nu)=\sup_{g\in \text{Lip}_1(\mathbb{R})} \Big| \int g d\mu -\int g d\nu \Big|,
}
we have 
\be{
\big| \E f''(W) -\N f'' \big|\leq C  \mathcal{W}_1 (\mathcal{L}(W), N(0,1))\leq C  \mathcal{W}_2 (\mathcal{L}(W), N(0,1)).
}
This proves \eq{4}.
\end{proof}

\subsubsection{$\mathcal{W}_2$ bound for approximating $\mathcal{L}(W)$ by the distribution of a sum of i.i.d.\ random variables}
Note that in proving Theorem \ref{t1}, we can assume that $|\beta|$ is smaller than an arbitrarily chosen constant $c_1>0$. 
If $\beta\ne 0$,
let $n=\lfloor c_2 \beta^{-2} \rfloor$ for a constant $c_2>0$ to be chosen.
Let $\{\xi_i: i=1,\dots, n\}$ be i.i.d.\ such that
\be{
\P(\xi_1=-\frac{3}{2})=\frac{3}{16}-\frac{\sqrt{n} \beta}{6},
} 
\be{
\P(\xi_1=-\frac{1}{2})=\frac{5}{16}+\frac{\sqrt{n} \beta}{2},
} 
\be{
\P(\xi_1=\frac{1}{2})=\frac{5}{16}-\frac{\sqrt{n} \beta}{2},
} 
\be{
\P(\xi_1=\frac{3}{2})=\frac{3}{16}+\frac{\sqrt{n} \beta}{6},
} 
where we choose $c_2$ to be small enough so that the above is indeed a probability distribution, and then choose $c_1$ to be small enough so that $n\geq 1$. By straightforward computation, we have
\be{
\E \xi_i=0,\  \E \xi_i^2=1,\  \E \xi_i^3=\sqrt{n} \beta,\  \E \xi_i^4\leq C.
}
Let $V_n=\frac{1}{\sqrt{n}}\sum_{i=1}^n \xi_i$.
Note that $\kappa_3(V_n)=\beta$, where $\kappa_r$ denotes the $r$th cumulant, and $\sum_{i=1}^n \frac{\E \xi_i^4}{n^2}\leq \frac{C}{n}\leq C \beta^2$.
The expansion in Theorem 1 of \cite{Ba86} implies
\ben{\label{5}
\Big|\E h(V_n)-\N h+\frac{\beta}{2} \N f''    \Big|\leq C \beta^2.
}
If $\beta=0$, let $V_n\sim N(0,1)$ and \eq{5} automatically holds.
From Lemma \ref{l1} and the expansions \eq{4} and \eq{5}, we have
\besn{\label{6}
&\mathcal{W}_2(\mathcal{L}(W), \mathcal{L}(V_n))\\
\leq &C \Big\{ \sup_{h\in \Lambda_2} \big[ \E h(W) -\E h (V_n)   \big] \Big\}^{\frac{1}{2}}\\
\leq &C\Big\{ |\beta|+\big[ |\beta| \mathcal{W}_2(\mathcal{L}(W), N(0,1))\big]^{\frac{1}{2}} +(\gamma_1+\gamma_2+\gamma_3)^{\frac{1}{2}}    \Big\}.
}

We remark that \cite{Ri09} used a Poisson-like approximation for $\mathcal{L}{(W)}$.
Approximating by sums of i.i.d.\ random variables enables us to potentially bound the Wasserstein-$p$ distance for any positive integer $p$.

\subsubsection{Triangle inequality and the final bound}
By Lemma \ref{l3},
\ben{\label{18}
\mathcal{W}_2(\mathcal{L}(V_n), N(0,1))\leq C \Big\{ \sum_{i=1}^n \frac{\E \xi_i^4}{n^2} \Big\}^{\frac{1}{2}}\leq C |\beta|.
}
Using the triangle inequality, \eq{6} and \eq{18}, we obtain
\bes{
&\mathcal{W}_2(\mathcal{L}(W), N(0,1))\\
\leq & \mathcal{W}_2(\mathcal{L}(W), \mathcal{L}(V_n)) +\mathcal{W}_2(\mathcal{L}(V_n), N(0,1))    \\
\leq & C\Big\{ |\beta|+\big[ |\beta| \mathcal{W}_2(\mathcal{L}(W), N(0,1))\big]^{\frac{1}{2}} +(\gamma_1+\gamma_2+\gamma_3)^{\frac{1}{2}}    \Big\}.
}
Finally, we use the inequality $\sqrt{ab}\leq \frac{1}{2\epsilon} a + \frac{\epsilon}{2} b$ with $a=|\beta|$ and $b=\mathcal{W}_2(\mathcal{L}(W), N(0,1))$, choose a sufficiently small $\epsilon$ and solve the recursive inequality for $\mathcal{W}_2(\mathcal{L}(W), N(0,1))$ to obtain the bound \eq{2}.

%\subsubsection{Possible extension}

\subsection{Proof of Corollary \ref{cor1}}
For each $i=1,\dots, n$, let $A_i=\{j: |j-i|\leq m\}$.
For each $i=1,\dots, n$ and $j\in A_i$, let $A_{ij}=\{k: \min\{|k-j|, |k-i|\}\leq m\}$.\
For each $i=1,\dots, n$, $j\in A_i$ and $k\in A_{ij}$, let $A_{ijk}=\{l: \min\{|l-i|, |l-j|, |l-k|\}\leq m\}$.
By the $m$-dependence assumption, they satisfy the assumptions (LD1)--(LD3) for Theorem \ref{t1}.
For the first term in the definition of $\beta$ of Theorem \ref{t1}, we have
\bes{
&|\sum_{i=1}^n \sum_{j,k\in A_i} \E X_i X_j X_k|\\
\leq & C\sum_{i=1}^n \sum_{j,k\in A_i} (\E|X_i|^3+\E|X_j|^3+\E|X_k|^3)\\
\leq& Cm^2 \sum_{i=1}^n \E |X_i|^3,
}
where the last inequality is from the fact that each $i$ is counted at most $Cm^2$ times in the previous expression.
The second term of $\beta$ has the same upper bound.
Similarly, for $\gamma_1$, we have
\bes{
&\sum_{i\in I} \sum_{j\in A_i}\sum_{k\in A_{ij}}\sum_{l\in A_{ijk}} \E |X_i X_j X_k X_l|\\
\leq &C \sum_{i\in I} \sum_{j\in A_i}\sum_{k\in A_{ij}}\sum_{l\in A_{ijk}} (\E|X_i|^4+\E|X_j|^4+\E|X_k|^4+\E|X_l|^4)\\
\leq & Cm^3 \sum_{i=1}^n \E |X_i|^4,
}
and $\gamma_2$ and $\gamma_3$ have the same upper bound.
This proves the corollary.

\subsection{Proof of Theorem \ref{t3}}

Consider the index set 
\be{
I=\{i=(i_1,\dots, i_m): 1\leq i_1<\dots<i_m\leq n\}.
}
For each $i\in I$, let 
$\xi_i=\sigma_n^{-1} h(X_{i_1},\dots, X_{i_m}).$
Then $W_n=\sum_{i\in I} \xi_i$.
For each $i\in I$, let 
\be{
A_i=\{j\in I: i\cap j\ne \emptyset\}.
}
For each $i\in I$ and $j\in A_i$, let 
\be{
A_{ij}=\{k\in I: k\cap(i\cup j)\ne \emptyset\}.
}
For each $i\in I$, $j\in A_i$ and $k\in A_{ij}$, let 
\be{
A_{ijk}=\{l\in I: l\cap(i\cup j\cup k)\ne \emptyset\}.
}
Then they satisfy the conditions (LD1)--(LD3) of Theorem \ref{t1}.
Moreover, the sizes of the neighborhoods are all bounded by $C n^{m-1}$.
Note that by the non-degeneracy condition,
$\sigma_n^2\asymp n^{2m-1}.$
By Theorem \ref{t1}, we have
\bes{
&\mathcal{W}_2(\mathcal{L}(W_n), N(0,1))\\
\leq& C\Big\{ n^m (n^{m-1})^2\frac{\E |h(X_1,\dots, X_m)|^3}{\sigma_n^3} + \big[ n^m (n^{m-1})^3 \frac{\E (h(X_1,\dots, X_m))^4}{\sigma_n^4}   \big]^{1/2}  \Big\}\\
\leq & C/\sqrt{n}.
}

\subsection{Proof of Theorem \ref{t2}}
In this subsection, the constants $C$ are allowed to depend on the given fixed graph $G$.
Let the potential edges of $K(n,p)$ be denoted by $(e_1,\dots, e_{{n\choose 2}})$.
Let $v=v(G), e=e(G)$. 
In applying Theorem \ref{t1}, let $W=\sum_{i\in I} X_i$,
where the index set is
\be{
I=\Big\{ i=(i_1,\dots, i_e): 1\leq i_1<\dots <i_e\leq {n\choose 2}, G_i:=(e_{i_1}, \dots, e_{i_e}) \ \text{is a copy of $G$}    \Big\},
}
\be{
X_i=\sigma^{-1} \big(  Y_i  -p^e\big),\quad \sigma^2:=\Var(S),\quad Y_i=\Pi_{l=1}^e E_{i_l},
}
and $E_{i_l}$ is the indicator of the event that the edge $e_{i_l}$ is connected in $K(n,p)$.
It is known that (cf. (3.7) of \cite{BaKaRu89})
\be{
\sigma^2\geq C (1-p) n^{2v} p^{2e} \psi^{-1}.
}
For each $i\in I$, let 
\be{
A_i=\{j\in I: e(G_j\cap  G_i)\geq 1\}.
}
For each $i\in I$ and $j\in A_i$, let
\be{
A_{ij}=\{k\in I: e(G_k\cap (G_i\cup  G_j))\geq 1\}.
}
For each $i\in I$, $j\in A_i$ and $k\in A_{ij}$, let
\be{
A_{ijk}=\{l\in I:  e(G_l\cap ( G_i \cup  G_j\cup  G_k  ))\geq 1\},
}
Then these constructions satisfy (LD1)--(LD3) of Section 2.1.
Note that the $Y$'s are all increasing functions of the $E$'s.
By the arguments leading to (3.8) of \cite{BaKaRu89}, we have
\bes{
&\gamma:=\gamma_1+\gamma_2+\gamma_3\\
\leq & \Big\{ \frac{C}{\sigma^4} \sum_{i\in I}\sum_{j\in A_i}\sum_{k\in A_{ij}} \sum_{l\in A_{ijk}} \E(Y_i Y_j Y_k Y_l) \Big\} 
 \wedge  \Big\{ \frac{C}{\sigma^4}\sum_{i\in I}\sum_{j\in A_i}\sum_{k\in A_{ij}} \sum_{l\in A_{ijk}} \E(1-Y_i)   \Big\}.
}
For $\frac{1}{2}<p<1$, the latter term directly yields the estimate
\bes{
\gamma\leq &C\sigma^{-4} n^{v} n^{3(v-2)} (1-p)\\
\leq & C n^{4v-6} (1-p) [n^{2v-2} (1-p)]^{-2}\\
\leq & C n^{-2} (1-p)^{-1}.
}
Let $\cong$ denote graph homomorphism. 
For $0<p\leq \frac{1}{2}$, the former term gives
\bes{
\gamma\leq & C\sigma^{-4} \sum_{H\subset G\atop e(H)\geq 1} \sum_{i,j\in I\atop G_i\cap G_j\cong H} 
\sum_{K\subset (G_i\cup G_j)\atop e(K)\geq 1}\sum_{k\in I\atop G_k\cap(G_i \cup G_j)=K} \\
&\qquad \qquad  \Big\{ \sum_{L\subset (G_i\cup G_j\cup G_k)\atop e(L)\geq 1}\sum_{l\in I\atop G_l\cap(G_i \cup G_j\cup G_k)=L} 
p^{4e-e(H)-e(K)-e(L)} \Big\}\\
\leq &C\sigma^{-4} \sum_{H\subset G\atop e(H)\geq 1} \sum_{i,j\in I\atop G_i\cap G_j\cong H} 
\sum_{K\subset (G_i\cup G_j)\atop e(K)\geq 1}\sum_{k\in I\atop G_k\cap(G_i \cup G_j)=K}\\
&\qquad \qquad \Big\{ \sum_{L\subset (G_i\cup G_j\cup G_k)\atop L\subset G_m \ \text{for some}\  m,  e(L)\geq 1} 
n^{v-v(L)} p^{4e-e(H)-e(K)-e(L)} \Big\}\\
\leq & C\sigma^{-4} \psi^{-1} n^v p^e 
\sum_{H\subset G\atop e(H)\geq 1} \sum_{i,j\in I\atop G_i\cap G_j\cong H} 
\sum_{K\subset (G_i\cup G_j)\atop e(K)\geq 1}\sum_{k\in I\atop G_k\cap(G_i \cup G_j)=K}
p^{3e- e(H)-e(K)}\\
\leq & C\sigma^{-2} (\psi^{-1} n^v p^e)^2,
}
where in the last step, we used (3.10) of \cite{BaKaRu89}.
This gives 
\be{
\gamma\leq C\psi^{-1}.
}
In summary, we have proved that $\gamma^{1/2}$ is bounded by the right-hand side of \eq{1}.
By a similar and simpler argument which is essentially the same as (3.10) of \cite{BaKaRu89}, we also have that
$|\beta|$ is bounded by the right-hand side of \eq{1}.
Theorem \ref{t2} is now proved by invoking Theorem \ref{t1}.

\subsection{Supporting arguments for Conjecture \ref{con1}}

We follow the proof of Theorem \ref{t1}, obtain higher-order expansions and use a more careful choice of sum of i.i.d.\ random variables as an intermediate approximation. We first consider the case $p=3$. 

Let $h\in \Lambda_3$.
Let $f:=f_h$ in \eq{33} be the solution to
\be{
f'(w)-wf(w)=h(w)-\N h.
}
From $h\in \Lambda_3$ and Lemma \ref{l2}, 
\ben{\label{32}
|f^{(3)}(x)-f^{(3)}(x)|\leq C|x-y|.
}
We further let $g:=g_{f''}$, defined by replacing $h$ by $f''$ on the right-hand side of \eq{33}, be the solution to
\be{
g'(w)-wg(w)=f''(w)-\mathcal{N}f''.
}
From $\frac{1}{C}f''_1\in \Lambda_2$ and Lemma \ref{l2}, we have
\be{
|g''(x)-g''(y)|\leq C|x-y|.
}
Denote the third cumulant of $W$ by
\be{
\kappa_3:=\kappa_3(W)=\sum_{i\in I} \sum_{j,k\in A_i} \E X_i X_j X_k +2\sum_{i\in I} \sum_{j\in A_i}\sum_{k\in A_{ij}\backslash A_i} \E X_i X_j X_k,
}
which we denoted by $\beta$ before.
Denote the fourth cumulant of $W$ by $\kappa_4:=\kappa_4(W)$.
%\bes{
%\kappa_4=&\sum_{i\in I} \E X_i (\sum_{j\in A_i} X_j)^3-3\sum_{i\in I} \sum_{j\in A_i} \E X_i X_j \E(\sum_{k\in A_{ij}}X_k)^2\\
%&+3\sum_{i\in I}\sum_{j\in A_i} \E X_i X_j (\sum_{k\in A_{ij}\backslash A_i}X_k)^2
%+3\sum_{i\in I} \sum_{j,k\in A_i} \sum_{l\in A_{ijk}\backslash A_i} \E X_i X_j X_k X_l\\
%&+6\sum_{i\in I} \sum_{j\in A_i} \sum_{k\in A_{ij}\backslash A_i} \sum_{l\in A_{ijk}\backslash A_{ij}}\E X_i X_j X_k X_l.
%}
A tedious but similar expansion as for \eq{25} yields
\besn{\label{27}
&\E h(W)-\N h=\E f'(W)-\E Wf(W)\\
=&-\frac{\kappa_3}{2} \E f''(W)-\frac{\kappa_4}{6} \E f^{(3)}(W) + O(R_3).
}
Since $\frac{1}{C}f''_1\in \Lambda_2$, from \eq{4}, we have
\ben{\label{28}
|\E f''(W)-\mathcal{N} f''+\frac{\kappa_3}{2} \mathcal{N} g''| \leq C\big[ |\kappa_3| \mathcal{W}_3(\mathcal{L}(W), N(0,1))+R_2\big].
}
From \eq{32}, we have
\ben{\label{29}
\E f^{(3)}(W)-\mathcal{N} f^{(3)}=O(\mathcal{W}_3(\mathcal{L}(W), N(0,1))).
}
From \eq{27}--\eq{29} and $|\kappa_3|\leq CR_1$, $|\kappa_4|\leq C R_2$, we have
\besn{\label{30}
&\Big| \E h(W)-\mathcal{N}h +\frac{\kappa_3}{2} \N f'' +\frac{\kappa_4}{6} \N f^{(3)} - \frac{\kappa_3^2}{4} \N g''\Big|\\ 
\leq & C\Big[(R_1^2+R_2)\mathcal{W}_3(\mathcal{L}(W), N(0,1))+R_1 R_2+R_3      \Big].
}
Note that the above expansion reduces to that of (8) of \cite{Ba86} for sums of independent random variables.

Without loss of generality, assume that $R_3$ and $R_4$, hence $|\kappa_3|$ and $|\kappa_4|$ are smaller than an arbitrarily chosen constant $c_1>0$.
Otherwise, the bound \eq{24} is trivial for $p=3$ by choosing a large enough $C_3$.
If $\kappa_3\ne 0$ or $\kappa_4\ne 0$,
let 
\be{
n= \floor{ c_2\kappa_3^{-2}} \wedge \floor{c_2 |\kappa_4|^{-1}}
}
for a constant $c_2>0$ to be chosen. 
Let $\{\xi_i: i=1,\dots, n\}$ be i.i.d.\ such that
\be{
\P(\xi_1=-2)=\frac{1}{12}+\frac{-2\sqrt{n} \kappa_3+n\kappa_4}{24},
} 
\be{
\P(\xi_1=-1)=\frac{1}{6}+\frac{\sqrt{n} \kappa_3-n\kappa_4}{6},
} 
\be{
\P(\xi_1=0)=\frac{1}{2}+\frac{n\kappa_4}{4},
} 
\be{
\P(\xi_1=1)=\frac{1}{6}-\frac{\sqrt{n} \kappa_3+n\kappa_4}{6},
} 
\be{
\P(\xi_1=2)=\frac{1}{12}+\frac{2\sqrt{n} \kappa_3+n\kappa_4}{24},
} 
where we choose $c_2$ to be small enough so that the above is indeed a probability distribution, and then choose $c_1$ to be small enough so that $n\geq 1$. By straightforward computation, we have
\be{
\E \xi_1=0, \ \E \xi_2^2=1,\ \kappa_3(\xi_1)=\sqrt{n} \kappa_3,\ \kappa_4(\xi_1)=n \kappa_4,\ \E|\xi_1|^5\leq C.
}
Let $V_n=\frac{1}{\sqrt{n}}\sum_{i=1}^n \xi_i$.
The expansion in Theorem 1 of \cite{Ba86} implies
\ben{\label{31}
\Big| \E h(V_n)-\mathcal{N}h +\frac{\kappa_3}{2} \N f'' +\frac{\kappa_4}{6} \N f^{(3)} - \frac{\kappa_3^2}{4} \N g''   \Big|\leq \frac{C}{n^{3/2}}
\leq C(R_1^3+R_2^{3/2}).
}
If $\kappa_3=\kappa_4= 0$, let $V_n\sim N(0,1)$ and \eq{31} automatically holds.
The expansions \eq{30} and \eq{31} imply
\be{
|\E h(W)-\E h(V_n)|\leq C\big[  (R_1^2+R_2)\mathcal{W}_3(\mathcal{L}(W), N(0,1))+R_1^3+R_2^{3/2}+R_3   \big],
}
where we used Young's inequality $|ab|\leq C(|a|^3+|b|^{3/2})$.
As in the proof of Theorem \ref{t1}, we have
\bes{
&\mathcal{W}_3(\mathcal{L}(W), N(0,1))\\
\leq & \mathcal{W}_3(\mathcal{L}(W), \mathcal{L}(V_n))+ C(R_1+R_2^{1/2})\\
\leq & C(R_1+R_2^{1/2}+R_3^{1/3})+C(R_1+R_2^{1/2})^{2/3}(\mathcal{W}_3(\mathcal{L}(W), N(0,1)))^{1/3}\\
\leq &\frac{1}{2}\mathcal{W}_3(\mathcal{L}(W), N(0,1))+C(R_1+R_2^{1/2}+R_3^{1/3}).
}
This implies the conjectured result for $p=3$.

For the case $p\geq 4$ and $h\in \Lambda_p$, we start with the expansion
\bes{
&\E h(W)-\N h=\E f'(W)-\E Wf(W)\\
=&-\sum_{m=1}^{p-1}\frac{\kappa_{m+2}}{(m+1)!} \E f^{(m)}(W)+ O(R_p),
}
where $f=f_h$ in \eq{33} is the solution to \eq{7} and $\kappa_{m+2}:=\kappa_{m+2}(W)$ is the $(m+2)$th cumulant of $W$.
To see that the coefficients must be of the given form of the cumulants, take $f(w)=w^2, w^3, \dots$ in the expansion.
The constraint that any pair of $\E$'s must be separated by at least two $X_i$'s is from the assumption that $\E X_i=0$ for any $i\in I$.
The conjectured result should then follow by similar arguments as for the case $p=3$.

\section*{Acknowledgements}

The author would like to thank Michel Ledoux for introducing the problem and Jia-An Yan for comments on an earlier version of this paper.
This work was partially supported by Hong Kong RGC ECS 24301617, a CUHK direct grant and a CUHK start-up grant.
%\bibliography{/Users/fangxiao/Dropbox/bibdesk/mybib}
%\bibliographystyle{apalike}

\end{document}